\documentclass{article}
\usepackage{amsmath}
\usepackage{amsfonts}
\usepackage{amssymb}
\usepackage{pstricks,pst-node,amssymb,amsmath,graphics,latexsym,tabularx,shapepar}
\usepackage{graphicx}
\newtheorem{theorem}{Theorem}

\newtheorem{corollary}[theorem]{Corollary}

\newtheorem{proposition}[theorem]{Proposition}

\newenvironment{proof}[1][Proof]{\noindent\textbf{#1.} }{\ \rule{0.5em}{0.5em}}
\newcommand{\bpartial}{\mathop{\partial\kern -4pt\raisebox{.8pt}{$|$}}}
\newcommand{\bra}{\mathopen{[\kern-1.6pt[}}
\newcommand{\ket}{\mathclose{]\kern-1.5pt]}}
\newcommand{\bbra}{\mathopen{[\kern-2.2pt[\kern-2.3pt[}}
\newcommand{\bket}{\mathclose{]\kern-2.1pt]\kern-2.3pt]}}

\title{Poisson–Hamiltonian Pontryagin Dynamics and Optimal Control of Mechanical Systems on Lie Groupoids}
\author{Ghorbanali Haghighatdoost \\
Department of Mathematics, Azarbaijan Shahid Madani University, Tabriz, Iran \\
Email: gorbanali@azaruniv.ac.ir}

\begin{document}

\maketitle
\begin{abstract}
We develop a Poisson--Hamiltonian formulation of Pontryagin dynamics for optimal control of mechanical systems on Lie groupoids. 
The reduced dynamics is formulated intrinsically on the dual Lie algebroid endowed with its canonical linear Poisson structure and evolves on its symplectic leaves. 

The main result of this work shows that symplectic leaves, rather than coadjoint orbits, provide the natural reduced phase spaces for Pontryagin dynamics on Lie groupoids.
Under suitable regularity assumptions, we prove the equivalence between the variational formulation of the optimal control problem and the associated Poisson--Hamiltonian Pontryagin system, and we show that groupoid-invariant Lagrangians lead to reduced optimality conditions of Euler--Poincar\'e type. 
Several mechanical examples, including systems with configuration-dependent inertia and local symmetries, are presented to illustrate the theory.
\end{abstract}



\section{Introduction}

Optimal control of mechanical systems with symmetry is classically formulated within the framework of Lie groups.
In this setting, the Pontryagin Maximum Principle leads to Hamiltonian dynamics on the dual Lie algebra endowed with its Lie--Poisson structure, and the reduced trajectories evolve on coadjoint orbits \cite{Kirillov2004, Vaisman1994, Weinstein1983, HaghighatdoostAyoubiCoadjoint}.
This Poisson--Hamiltonian viewpoint has been developed systematically for Lie group symmetries, notably in the work of Bloch and Crouch, and provides a geometric interpretation of reduced optimal control dynamics.\cite{BlochCrouch1999, Bloch,  BulloLewis, Hagh}

However, many mechanical systems of interest do not possess global Lie group symmetries. 
Systems with configuration-dependent, local, or partially defined symmetries arise naturally in mechanics, robotics, and constrained dynamics. 
Such systems are more naturally described by Lie groupoids and Lie algebroids, which generalize Lie groups by allowing symmetry to depend on the base configuration. 
In this broader setting, coadjoint orbits are no longer the fundamental geometric objects governing reduced dynamics.

From a geometric perspective, the dual of a Lie algebroid carries a canonical linear Poisson structure. 
The Hamiltonian dynamics associated with this structure is organized by its symplectic leaves. 
Consequently, in the Lie groupoid setting, \emph{symplectic leaves replace coadjoint orbits as the natural reduced phase spaces} for Hamiltonian and optimal control dynamics. 
Coadjoint orbits appear only as a special case corresponding to Lie group symmetries, whereas symplectic leaves provide an intrinsic and fully general description \cite{MarsdenRatiu1999, CendraMarsdenRatiu2001, HolmMarsdenRatiu, HaghighatdoostEPGroupoids, Martinez2001, Weinstein1983}.

Variational mechanics on Lie groupoids and Lie algebroids has been investigated in several foundational works \cite{Weinstein1996, Bloch,CendraMarsdenRatiu2001, HaghighatdoostEPGroupoids, HolmMarsdenRatiu, Landsman, Martinez2001}, providing the appropriate geometric framework beyond Lie groups. 
Nevertheless, until recently, Euler--Poincar\'e equations had not been formulated explicitly and systematically at the level of Lie groupoids. 

In our previous work, Euler--Poincar\'e equations on Lie groupoids were derived directly and intrinsically, yielding the first clear formulation of Euler--Poincar\'e equations as genuine groupoid-level dynamical equations.

In subsequent work, we developed a Hamiltonian framework for mechanical systems on coadjoint Lie groupoids and analyzed the associated Hamiltonian dynamics. 
These results clarified how classical Hamiltonian reduction on coadjoint orbits can be extended to certain groupoid settings, while also revealing the limitations of orbit-based descriptions in the absence of global Lie group symmetries.

The purpose of the present paper is to build directly on these developments and to connect them with optimal control theory. 
Starting from a variational optimal control problem on a Lie algebroid, we derive an intrinsic Poisson--Hamiltonian formulation of Pontryagin dynamics on the dual Lie algebroid. 
We show that the resulting optimal dynamics is governed by the canonical Poisson structure and is naturally confined to the symplectic leaves of the dual Lie algebroid.

Moreover, when the Lie algebroid integrates to a Lie groupoid and the Lagrangian is invariant under the groupoid action, the necessary conditions for optimality admit a reduced description in terms of controlled Euler--Poincar\'e equations on Lie groupoids. 
This establishes a precise link between Euler--Poincar\'e reduction on Lie groupoids, Poisson--Hamiltonian dynamics, and Pontryagin's Maximum Principle.

The framework developed here extends Poisson--Hamiltonian optimal control from Lie groups to Lie groupoids and provides a natural geometric setting for mechanical systems with local or configuration-dependent symmetries, where classical coadjoint-orbit-based reductions are no longer sufficient.

The variational formulation of the optimal control problem considered in this paper may be interpreted as a Lagrange--D'Alembert principle on a Lie groupoid.
Under groupoid invariance, the reduction procedure developed in our work on Euler--Poincar\'e equations on Lie groupoids yields controlled Euler--Poincar\'e equations on the Lie algebroid.
The present results show that these reduced variational equations are equivalent, under regularity assumptions, to a Poisson--Hamiltonian formulation on the dual Lie algebroid $A^*$.

\section{Poisson Geometry and Reduced Optimal Control}

Reduced mechanical systems with symmetry are naturally described by Poisson structures rather than canonical symplectic forms.  
This viewpoint is well established in geometric mechanics, where symmetry reduction typically leads to Hamiltonian systems evolving on Poisson manifolds whose symplectic leaves encode the reduced dynamics.

In the classical Lie group setting, the reduced phase space is the dual Lie algebra $\mathfrak g^*$ endowed with the Lie--Poisson bracket.  
The symplectic leaves of this Poisson structure coincide with the coadjoint orbits of the group, each carrying the Kirillov--Kostant--Souriau symplectic form.  
Within this framework, optimal control problems for mechanical systems admit a natural Hamiltonian formulation, as shown for example in the work of Bloch and Crouch, where Pontryagin's Maximum Principle is interpreted as a Hamiltonian flow on $\mathfrak g^*$.

However, global Lie group symmetries are often too restrictive for mechanical systems with configuration--dependent or local symmetries.  
In such cases, the appropriate infinitesimal structure is not a Lie algebra but a Lie algebroid $A \to M$.  
The dual bundle $A^*$ carries a canonical linear Poisson structure, which plays the role of the Lie--Poisson structure in the group case and provides the correct reduced phase space for Hamiltonian dynamics.

As shown in our earlier work on Hamiltonian systems on coadjoint Lie groupoids, the symplectic leaves of this Poisson manifold may, under suitable integrability conditions, be realized as coadjoint orbits associated with the underlying Lie groupoid.  
In general, however, these symplectic leaves constitute the primary geometric objects governing the reduced dynamics, independently of whether a global groupoid action is explicitly realized.

The purpose of the present work is to formulate optimal control problems for mechanical systems on Lie groupoids directly within this Poisson--Hamiltonian framework.  
In particular, we aim to develop a Pontryagin-type Hamiltonian formulation on the dual Lie algebroid $A^*$ and to show that the resulting optimal dynamics are constrained to the symplectic leaves determined by the canonical Poisson structure.

\section{Poisson Structure on the Dual Lie Algebroid}

Let $A \to M$ be a Lie algebroid with anchor map $\rho : A \to TM$ and Lie bracket $[\cdot,\cdot]$ on the space of sections $\Gamma(A)$.  
It is well known that the dual bundle $A^*$ carries a canonical linear Poisson structure, which generalizes the Lie--Poisson structure on the dual of a Lie algebra.

The Poisson bracket on $C^\infty(A^*)$ is uniquely determined by the following properties:
\begin{itemize}
\item For $f,g \in C^\infty(M)$, viewed as functions on $A^*$ by pullback,
\[
\{f,g\} = 0.
\]
\item For $\sigma \in \Gamma(A)$ and $f \in C^\infty(M)$,
\[
\{\ell_\sigma,f\} = \rho(\sigma)(f),
\]
where $\ell_\sigma(\alpha_x) = \langle \alpha_x,\sigma(x)\rangle$ denotes the linear function associated with $\sigma$.
\item For $\sigma,\eta \in \Gamma(A)$,
\[
\{\ell_\sigma,\ell_\eta\} = \ell_{[\sigma,\eta]}.
\]
\end{itemize}

This Poisson structure is linear in the fibers of $A^*$ and reduces, in the special case where $A=\mathfrak g$ is a Lie algebra, to the classical Lie--Poisson bracket on $\mathfrak g^*$.

The Hamiltonian vector field $X_H$ associated with a Hamiltonian function $H \in C^\infty(A^*)$ is defined by
\[
X_H(F) = \{F,H\}, \qquad F \in C^\infty(A^*).
\]
In local coordinates $(x^i,\mu_a)$ on $A^*$, where $\{e_a\}$ is a local basis of sections of $A$, the Poisson bracket takes the form
\[
\{x^i,x^j\} = 0, \qquad
\{x^i,\mu_a\} = \rho_a^i(x), \qquad
\{\mu_a,\mu_b\} = C_{ab}^c(x)\mu_c,
\]
where $\rho_a^i$ are the components of the anchor and $C_{ab}^c$ are the structure functions of the Lie algebroid.

The symplectic leaves of this Poisson manifold are immersed submanifolds of $A^*$ obtained as the integral manifolds of the characteristic distribution generated by Hamiltonian vector fields.  
These leaves generalize coadjoint orbits of Lie groups and provide the natural phase spaces for reduced Hamiltonian dynamics on Lie groupoids.

In the presence of an integrating Lie groupoid $\mathcal G \rightrightarrows M$, these symplectic leaves may be realized, under suitable regularity assumptions, as coadjoint orbits associated with $\mathcal G$, in the sense described in our earlier works \cite{Hagh,HaghighatdoostAyoubiCoadjoint}.  
However, the Poisson structure on $A^*$ and its symplectic foliation are defined intrinsically at the algebroid level and do not depend on the existence of a global groupoid action.

\section{Optimal Control Systems on Lie Algebroids}

Let $A \to M$ be a Lie algebroid with anchor $\rho : A \to TM$.  
An \emph{admissible curve} in $A$ is a curve $a(t) \in A$ with base curve $x(t)=\pi(a(t))$ satisfying
\[
\dot{x}(t) = \rho(a(t)).
\]
Such curves generalize trajectories of left- or right-invariant systems on Lie groups and describe the kinematics of mechanical systems with local symmetries.

\subsection{Control systems}

Let $E \to M$ be a control bundle and let
\[
\Phi : E \to A
\]
be a smooth bundle map covering the identity on $M$.  
A control $u(t) \in E_{x(t)}$ determines an admissible curve
\[
a(t) = \Phi(u(t)) \in A_{x(t)},
\]
and the induced state equation is
\[
\dot{x}(t) = \rho(\Phi(u(t))).
\]
This defines an optimal control system on the Lie algebroid $A$.

\subsection{Cost functional}

Let
\[
L : E \to \mathbb{R}
\]
be a smooth cost function.  
Given fixed endpoints $x(0)=x_0$ and $x(T)=x_T$, the optimal control problem consists in minimizing the functional
\[
J(u) = \int_0^T L(u(t))\,dt
\]
over all admissible controls $u(t)$.

This formulation includes, as special cases:
\begin{itemize}
\item Optimal control on Lie groups (when $M$ is a point),
\item Euler--Poincaré optimal control systems,
\item Optimal control problems on coadjoint Lie groupoids.
\end{itemize}

\subsection{Pontryagin Hamiltonian}

Consider the Pontryagin bundle
\[
A^* \times_M E = \{(\alpha_x,u_x) \mid \alpha_x \in A_x^*,\; u_x \in E_x\}.
\]
The Pontryagin Hamiltonian is defined by
\[
H(\alpha_x,u_x) = \langle \alpha_x, \Phi(u_x) \rangle - L(u_x).
\]

For fixed $(\alpha_x,x)$, the optimal control $u^*(\alpha_x)$ is determined by the stationarity condition
\[
\frac{\partial H}{\partial u}(\alpha_x,u^*(\alpha_x)) = 0,
\]
whenever this equation admits a solution.

Substituting $u^*$ yields the reduced Hamiltonian
\[
\mathcal{H}(\alpha_x) = H(\alpha_x,u^*(\alpha_x)),
\]
which is a smooth function on $A^*$.

\subsection{Hamiltonian dynamics on $A^*$}

The dual Lie algebroid $A^*$ is equipped with its canonical linear Poisson structure.  
The Hamiltonian vector field associated with $\mathcal{H}$ is defined by
\[
\dot{F} = \{F,\mathcal{H}\}, \qquad F \in C^\infty(A^*),
\]
where $\{\cdot,\cdot\}$ denotes the Poisson bracket on $A^*$.

In local coordinates $(x^i,\mu_a)$ on $A^*$, the equations of motion take the form
\[
\dot{x}^i = \rho_a^i(x)\frac{\partial \mathcal{H}}{\partial \mu_a}, \qquad
\dot{\mu}_a = -\rho_a^i(x)\frac{\partial \mathcal{H}}{\partial x^i}
- C_{ab}^c(x)\mu_c \frac{\partial \mathcal{H}}{\partial \mu_b},
\]
where $\rho_a^i$ and $C_{ab}^c$ are the anchor and structure functions of the Lie algebroid.

\subsection{Geometric properties}

The Hamiltonian flow generated by $\mathcal{H}$ preserves the Poisson structure on $A^*$ and is therefore confined to the symplectic leaves of $A^*$.  
These leaves generalize coadjoint orbits in the Lie group case and constitute the natural reduced phase spaces for optimal control dynamics.

When the Lie algebroid integrates to a Lie groupoid, these symplectic leaves may be realized as coadjoint orbits of the groupoid in the sense developed in earlier work.  
However, the Hamiltonian formulation above is intrinsic and does not depend on the explicit realization of such orbits.

\paragraph{Relation to previous work.}
Optimal control problems on coadjoint Lie groupoids were previously investigated by the author \cite{Hagh}, where Hamiltonian dynamics were formulated on coadjoint orbits associated with a Lie groupoid.  
The present work adopts a different but compatible perspective: rather than starting from orbit realizations, we formulate optimal control intrinsically on the dual Lie algebroid $A^*$ equipped with its canonical linear Poisson structure.  
This approach clarifies the geometric origin of the Pontryagin Hamiltonian dynamics, identifies symplectic leaves as the fundamental reduced phase spaces, and provides a direct generalization of the Poisson--Hamiltonian framework for optimal control developed by Bloch and Crouch in the Lie group setting.

\section{Two Optimal Control Problems}

In this section we formulate and compare two optimal control problems associated with mechanical systems on a Lie algebroid.  
The first problem is posed in a variational (Lagrangian) framework on the Lie algebroid itself, while the second problem is formulated in a Hamiltonian framework on the dual Lie algebroid using Pontryagin's Maximum Principle.  
We show that, under natural regularity assumptions, these two formulations are equivalent and lead to the same reduced Hamiltonian dynamics on the Poisson manifold $A^*$.

\subsection{Problem I: Variational optimal control on a Lie algebroid}

Let $A \to M$ be a Lie algebroid with anchor $\rho : A \to TM$.  
Consider admissible curves $a(t) \in A$ with base curve $x(t)=\pi(a(t))$ satisfying
\[
\dot{x}(t) = \rho(a(t)).
\]

Let $E \to M$ be a control bundle and $\Phi : E \to A$ a smooth bundle map.  
The controlled dynamics are given by
\[
a(t) = \Phi(u(t)), \qquad u(t) \in U_{x(t)}.
\]

Let
\[
L : E \to \mathbb{R}
\]
be a smooth Lagrangian.  
The variational optimal control problem consists in minimizing
\[
J(u) = \int_0^T L(u(t))\,dt
\]
over all admissible controls $u(t)$ subject to fixed endpoint conditions on $x(t)$.

Critical points of this problem are characterized by constrained variations of admissible curves in the Lie algebroid.  
The resulting necessary conditions take the form of controlled Euler--Lagrange equations on $A$, or equivalently, Euler--Poincaré-type equations when symmetry reduction is present.

\subsection{Problem II: Hamiltonian optimal control on the dual Lie algebroid}

The same optimal control problem may be reformulated using Pontryagin's Maximum Principle.  
Consider the Pontryagin bundle
\[
A^* \times_M E
\]
and define the Pontryagin Hamiltonian
\[
H(\alpha_x,u_x) = \langle \alpha_x, \Phi(u_x) \rangle - L(u_x),
\qquad (\alpha_x,u_x) \in A^* \times_M E.
\]

Assuming the stationarity condition
\[
\frac{\partial H}{\partial u}(\alpha_x,u_x) = 0
\]
admits a solution $u^*(\alpha_x)$, we define the reduced Hamiltonian
\[
\mathcal{H}(\alpha_x) = H(\alpha_x,u^*(\alpha_x)).
\]

The optimal dynamics are then governed by the Hamiltonian vector field generated by $\mathcal{H}$ with respect to the canonical linear Poisson structure on $A^*$:
\[
\dot{F} = \{F,\mathcal{H}\}, \qquad F \in C^\infty(A^*).
\]

This yields a Poisson--Hamiltonian system whose trajectories evolve on symplectic leaves of $A^*$.

\subsection{Equivalence of the two formulations}

Assume that the Lagrangian $L$ is regular in the control variables and that the map $\Phi : E \to A$ is fiberwise surjective.  
Then the Legendre transformation associated with $L.$

\section{Main Equivalence Theorem}

We now state the central result of this work, establishing the equivalence between the variational and Hamiltonian formulations of optimal control for mechanical systems on Lie algebroids.

\begin{theorem}[Equivalence of variational and Hamiltonian optimal control]
Let $A \to M$ be a Lie algebroid with anchor $\rho : A \to TM$, and let
\[
\Phi : E \to A
\]
be a smooth fiberwise surjective bundle map defining a control system on $A$.  
Let $L : E \to \mathbb{R}$ be a smooth Lagrangian, regular in the control variables.

Then the following statements hold:

\begin{enumerate}
\item[(i)] Every critical trajectory of the variational optimal control problem on $A$, defined by minimizing
\[
J(u) = \int_0^T L(u(t))\,dt
\]
over admissible curves, projects to an integral curve of the Hamiltonian vector field generated by the reduced Pontryagin Hamiltonian
\[
\mathcal{H}(\alpha_x) = \langle \alpha_x, \Phi(u^*(\alpha_x)) \rangle - L(u^*(\alpha_x))
\]
on the dual Lie algebroid $A^*$, where $u^*$ is determined by the stationarity condition.

\item[(ii)] Conversely, every solution of the Hamiltonian system on $A^*$ generated by $\mathcal{H}$ projects to a critical trajectory of the variational optimal control problem on $A$.

\item[(iii)] The Hamiltonian flow generated by $\mathcal{H}$ is a Poisson flow with respect to the canonical linear Poisson structure on $A^*$, and its trajectories are confined to symplectic leaves of $A^*$.

\end{enumerate}

Consequently, the variational and Hamiltonian formulations of the optimal control problem on the Lie algebroid $A$ are equivalent.
\end{theorem}
\begin{proof}[Proof sketch]
We outline the standard argument.

\medskip
\noindent\textbf{Step 1 (Pontryagin bundle and stationarity).}
Consider the Pontryagin bundle $A^*\times_M E$ and the Pontryagin Hamiltonian
\[
H(\alpha_x,u_x)=\langle \alpha_x,\Phi(u_x)\rangle - L(u_x).
\]
Regularity of $L$ in the control variables implies that the stationarity condition
\[
\partial_u H(\alpha_x,u_x)=0
\]
determines (locally, and globally under appropriate convexity) a smooth feedback $u^*(\alpha_x)$. Hence the reduced Hamiltonian
\[
\mathcal H(\alpha_x)=H(\alpha_x,u^*(\alpha_x))
\]
is well-defined on $A^*$.

\medskip
\noindent\textbf{Step 2 (Variational problem $\Rightarrow$ Hamiltonian system).}
Let $u(t)$ be a critical control for the variational problem and set $a(t)=\Phi(u(t))$ with base curve $x(t)$. Introduce a Lagrange multiplier curve $\alpha(t)\in A^*_{x(t)}$ enforcing the admissibility constraint $\dot x=\rho(a)$ and consider the augmented functional
\[
\widetilde J(u,\alpha)=\int_0^T\Big(L(u(t))-\langle \alpha(t),\Phi(u(t))\rangle
+\langle \alpha(t),a(t)\rangle\Big)\,dt
=\int_0^T\big(-H(\alpha(t),u(t))+\langle \alpha(t),a(t)\rangle\big)\,dt.
\]
Taking first variations with respect to $u$ gives the stationarity condition
$\partial_u H(\alpha(t),u(t))=0$, hence $u(t)=u^*(\alpha(t))$.
Variations compatible with the Lie algebroid constraint yield the adjoint equations; equivalently, the pair $(x(t),\alpha(t))$ satisfies Hamilton's equations on the Poisson manifold $A^*$ generated by $\mathcal H$:
\[
\dot F(t)=\{F,\mathcal H\}(x(t),\alpha(t))\qquad \forall\,F\in C^\infty(A^*).
\]
In local coordinates $(x^i,\mu_a)$ this reproduces the standard algebroid Hamilton equations
\[
\dot{x}^i=\rho_a^i(x)\frac{\partial\mathcal H}{\partial \mu_a},\qquad
\dot{\mu}_a=-\rho_a^i(x)\frac{\partial\mathcal H}{\partial x^i}
-C_{ab}^c(x)\mu_c\frac{\partial\mathcal H}{\partial \mu_b}.
\]

\medskip
\noindent\textbf{Step 3 (Hamiltonian system $\Rightarrow$ variational criticality).}
Conversely, let $\alpha(t)$ solve the Hamiltonian system on $A^*$ generated by the Pontryagin Hamiltonian $H$ and define $u(t)=u^*(\alpha(t))$. Then the triple $(x(t),u(t),\alpha(t))$ satisfies the PMP stationarity condition together with the state/adjoint dynamics, that is, the Pontryagin differential equations for the state and costate variables associated with $H$.


The PMP conditions are precisely the first-order necessary conditions for criticality of the original constrained variational problem, hence $u(t)$ is a critical control and $x(t)$ a critical trajectory.

\medskip
\noindent\textbf{Step 4 (Poisson invariance and symplectic leaves).}
Since $X_{\mathcal H}$ is a Hamiltonian vector field for the Poisson bracket on $A^*$, its flow preserves the Poisson tensor. Therefore trajectories remain in the integral manifolds of the characteristic distribution, i.e.\ in the symplectic leaves of $A^*$.

\medskip
This establishes the equivalence between the variational and Hamiltonian formulations, and the leafwise confinement of optimal dynamics.
\end{proof}
\begin{corollary}[Euler--Poincar\'e form of the optimality conditions on Lie groupoids]
Assume that the Lie algebroid $A$ integrates to a Lie groupoid $\mathcal G\rightrightarrows M$ and that the Lagrangian $L$ is $\mathcal G$--invariant.
Then, as we showed in \cite{HaghighatdoostEPGroupoids} the variational optimal control problem (Problem~I) reduces to a controlled Euler--Poincar\'e equation on the Lie algebroid,
\[
\frac{d}{dt}\frac{\partial \ell}{\partial X_t}
+ \operatorname{ad}^*_{X_t}\!\left(\frac{\partial \ell}{\partial X_t}\right)
= \mathcal F(u(t)),
\]
where $\ell$ is the reduced Lagrangian and $\mathcal F(u(t))$ is the force term induced by the control.
Moreover, if $\ell$ is regular, this controlled Euler--Poincar\'e equation is equivalent, via the Legendre transform, to the Poisson--Hamiltonian Pontryagin dynamics on $A^*$ generated by the reduced Hamiltonian $\mathcal H$(Problem~II).
\end{corollary}

\begin{corollary}
If $A=\mathfrak g$ is a Lie algebra, the above theorem reduces to the equivalence between variational and Pontryagin formulations of optimal control on Lie groups, as developed by Bloch and Crouch.  
In this case, the symplectic leaves of $A^*$ are the coadjoint orbits of the Lie group.
\end{corollary}
\begin{corollary}
If the Lie algebroid $A$ integrates to a Lie groupoid $\mathcal{G} \rightrightarrows M$, the Hamiltonian dynamics generated by $\mathcal{H}$ may be realized on coadjoint orbits of $\mathcal{G}$, recovering the optimal control formulation on coadjoint Lie groupoids developed in earlier work \cite{Hagh}.
\end{corollary}

\section{Examples}
\begin{itemize}

\item {An example beyond the Lie group case}

Consider the action Lie groupoid $\mathcal G = G \ltimes M \rightrightarrows M$
associated with a smooth action of a Lie group $G$ on a manifold $M$.
The corresponding Lie algebroid is $A=M\times\mathfrak g$ and the dual bundle
$A^*=M\times\mathfrak g^*$ carries a linear Poisson structure coupling the base
dynamics on $M$ with the Lie--Poisson structure on $\mathfrak g^*$.
An optimal control problem with cost
\[
L(x,u)=\tfrac12\|u\|^2+V(x)
\]
leads to Pontryagin dynamics on $A^*$ whose trajectories are not confined to
coadjoint orbits of $G$, but instead evolve on symplectic leaves depending on
the base point $x\in M$.
This illustrates a fundamental feature of the groupoid setting that is absent
in the classical Lie group framework of Bloch and Crouch.
\item{{\bf The Trivial Lie Groupoid $M\times G\times M$}}

Let $\mathcal G = M\times G\times M \rightrightarrows M$ with
$s(x,g,y)=y$, $t(x,g,y)=x$ and multiplication $(x,g,y)\circ(y,h,z)=(x,gh,z)$.
Its Lie algebroid identifies with $A\simeq TM\oplus (M\times\mathfrak g)$,
with anchor $\rho(X,\xi)=X$ and bracket
$[(X,\xi),(Y,\eta)]=([X,Y],X\cdot\eta-Y\cdot\xi+[\xi,\eta])$.
Hence $A^*\simeq T^*M\oplus (M\times\mathfrak g^*)$ carries the canonical linear Poisson structure.

Consider the optimal control problem with running cost
\[
L(x,v,\xi)=\tfrac12\|v\|^2+\tfrac12\langle \xi,\mathbb I(x)\xi\rangle+V(x),
\]
where $\mathbb I(x)$ is a configuration-dependent inertia operator.
The Pontryagin Hamiltonian
\[
H=\langle p,v\rangle+\langle \mu,\xi\rangle - L(x,v,\xi)
\]
yields the feedback $v^*=p^\sharp$ and $\xi^*=\mathbb I(x)^{-1}\mu$, and thus the reduced Hamiltonian
\[
\mathcal H(x,p,\mu)=\tfrac12\|p\|^2+\tfrac12\langle \mu,\mathbb I(x)^{-1}\mu\rangle - V(x).
\]
The corresponding Pontryagin dynamics is a Poisson--Hamiltonian flow on $A^*$; in particular,
\[
\dot x=\partial_p\mathcal H,\qquad 
\dot\mu=\operatorname{ad}^*_{\partial_\mu\mathcal H}\mu,
\qquad
\dot p=-\partial_x\mathcal H,
\]
showing a base--momentum coupling absent in the Lie group case of Bloch and Crouch.

For fiberwise constant $\xi$, the symplectic leaves are (locally) of the form $T^*M\times\mathcal O$, where $\mathcal O\subset\mathfrak g^*$ is a coadjoint orbit, with leafwise symplectic form $\omega_{T^*M}+\omega_{\mathrm{KKS}}$.

\paragraph{Case $G=\mathrm{SO}(3)$.}
Using $\mathfrak{so}(3)^*\simeq\mathbb R^3$, coadjoint orbits are spheres $\mathcal O_r=\{\mathbf m:\|\mathbf m\|=r\}$.
Consider the running cost
\[
L(x,v,\omega)=\tfrac12\|v\|^2+\tfrac12\,\omega^\top I(x)\omega+V(x),
\]
where $I(x)$ is a configuration-dependent inertia matrix.
The reduced Pontryagin Hamiltonian on $A^*$ becomes
\[
\mathcal H(x,p,\mathbf m)=\tfrac12\|p\|^2+\tfrac12\,\mathbf m^\top I(x)^{-1}\mathbf m - V(x),
\]
and the resulting Pontryagin dynamics is the Poisson--Hamiltonian system
\[
\dot x=\partial_p\mathcal H,\qquad
\dot{\mathbf m}=\mathbf m\times \partial_{\mathbf m}\mathcal H=\mathbf m\times(I(x)^{-1}\mathbf m),\qquad
\dot p=-\partial_x\mathcal H.
\]
This exhibits a base--momentum coupling absent in the Lie group setting.

\paragraph{Steering on $M=S^2$ (application).}
When $M=S^2$, the phase space $A^*\simeq T^*S^2\times \mathbb R^3$ and optimal trajectories evolve on symplectic leaves $T^*S^2\times \mathcal O_r$.
This provides a natural framework for steering problems on $S^2$ with internal $\mathrm{SO}(3)$ actuation, where the internal angular momentum remains confined to a coadjoint orbit while the base dynamics evolves on $S^2$.
\begin{proposition}[Confinement to coadjoint orbits in the $\mathrm{SO}(3)$ case]{\label{CCO}}
Let $G=\mathrm{SO}(3)$ and identify $\mathfrak{so}(3)^*\simeq\mathbb R^3$.
Assume the reduced Pontryagin dynamics on $A^*\simeq T^*M\oplus(M\times\mathbb R^3)$ satisfies
\[
\dot{\mathbf m}=\mathbf m\times \Omega,\qquad \Omega=\partial_{\mathbf m}\mathcal H(x,p,\mathbf m),
\]
for a smooth reduced Hamiltonian $\mathcal H$.
Then $\|\mathbf m(t)\|$ is constant along trajectories. In particular, $\mathbf m(t)$ remains on the coadjoint orbit
$\mathcal O_r=\{\mathbf m\in\mathbb R^3:\|\mathbf m\|=r\}$ determined by $\|\mathbf m(0)\|$.
\end{proposition}

\begin{proof}
Differentiate $\|\mathbf m\|^2=\mathbf m\cdot \mathbf m$:
\[
\frac{d}{dt}\|\mathbf m\|^2 = 2\,\mathbf m\cdot \dot{\mathbf m}
=2\,\mathbf m\cdot (\mathbf m\times \Omega)=0,
\]
since $\mathbf m\times \Omega$ is orthogonal to $\mathbf m$. Hence $\|\mathbf m(t)\|$ is constant.
\end{proof}
\item{{\bf A computable example on $M=S^2$ with $G=\mathrm{SO}(3)$}}

Let $M=S^2\subset\mathbb R^3$ and identify $\mathfrak{so}(3)^*\simeq\mathbb R^3$.
Consider the configuration-dependent inertia
\[
I(x)=I_0+\alpha\, x x^\top,\qquad I_0>0,\ \alpha\ge 0,
\]
and the potential $V(x)=\kappa(1-x_3)$.
The reduced Hamiltonian on $A^*\simeq T^*S^2\times\mathbb R^3$ is
\[
\mathcal H(x,p,\mathbf m)=\tfrac12\|p\|^2+\tfrac12\,\mathbf m^\top I(x)^{-1}\mathbf m - V(x).
\]
Setting $w=I(x)^{-1}\mathbf m$, the Pontryagin dynamics reads
\[
\dot x = (I-xx^\top)p,\qquad
\dot{\mathbf m}=\mathbf m\times w,
\]
and
\[
\dot p = (I-xx^\top)\Big(\alpha(x\cdot w)w + \nabla V(x)\Big),
\qquad \nabla V(x)=(0,0,-\kappa).
\]
By Proposition~\ref{CCO}, $\|\mathbf m(t)\|$ is constant, hence $\mathbf m(t)$ remains on the coadjoint orbit $\mathcal O_r$.
This example exhibits a base--momentum coupling through $I(x)$ that does not appear in the Lie group case.

{\bf Application: minimum-energy steering on $S^2$ with internal $\mathrm{SO}(3)$ actuation}

Let $M=S^2\subset\mathbb R^3$ and identify $\mathfrak{so}(3)^*\simeq\mathbb R^3$.
Consider the trivial Lie groupoid $S^2\times \mathrm{SO}(3)\times S^2\rightrightarrows S^2$, with Lie algebroid
$A\simeq TS^2\oplus(S^2\times\mathfrak{so}(3))$ and dual $A^*\simeq T^*S^2\oplus(S^2\times\mathbb R^3)$.
We study the minimum-energy steering problem: given $x_0,x_T\in S^2$ and a fixed horizon $T>0$, minimize
\[
J=\int_0^T\Big(\tfrac12\|v(t)\|^2+\tfrac12\,\omega(t)^\top I(x(t))\omega(t)+V(x(t))\Big)\,dt
\]
subject to $\dot x(t)=v(t)$, $x(0)=x_0$, $x(T)=x_T$, where
\[
I(x)=I_0+\alpha\,x x^\top,\qquad I_0>0,\ \alpha\ge0,
\qquad
V(x)=\kappa\big(1-x\cdot x_T\big).
\]
The reduced Hamiltonian on $A^*$ is
\[
\mathcal H(x,p,\mathbf m)=\tfrac12\|p\|^2+\tfrac12\,\mathbf m^\top I(x)^{-1}\mathbf m - V(x),
\qquad w=I(x)^{-1}\mathbf m,
\]
and the corresponding Pontryagin dynamics is the Poisson--Hamiltonian system
\[
\dot x=(I-xx^\top)p,\qquad
\dot{\mathbf m}=\mathbf m\times w,\qquad
\dot p=(I-xx^\top)\Big(\alpha(x\cdot w)w + \nabla V(x)\Big),
\]
with $\nabla V(x)=-\kappa\big(x_T-(x\cdot x_T)x\big)$.
Moreover, $\|\mathbf m(t)\|$ is conserved; hence $\mathbf m(t)$ remains on the coadjoint orbit $\mathcal O_r$ and optimal trajectories evolve on symplectic leaves $T^*S^2\times\mathcal O_r$.
This illustrates a base--momentum coupling through the configuration-dependent inertia $I(x)$, a feature absent in the classical Lie group setting.

\item{\bf Biomathematics Example: Spatially Structured Population Dynamics with Local Symmetries}

Spatially structured population dynamics provide a natural setting where Lie groupoids are needed. Consider a population distributed across multiple discrete spatial patches, with local interactions and transitions constrained by patch-specific rules. Global Lie group symmetries do not exist in this setting, but local symmetries between patches naturally form a Lie groupoid structure.

Let $x_i(t)$ denote the population density in patch $i$, and $u_i(t)$ a control input, such as migration rate or harvesting effort. The control system can be formulated on a Lie algebroid $A \to M$, with control bundle $E \to M$, where $M$ represents the set of patches and $A$ encodes allowed transitions and interactions between them.

The optimal control problem is then to maximize a suitable objective functional, for instance total yield or survival probability, subject to the local dynamics and constraints. Applying Pontryagin's Maximum Principle on the dual bundle $A^*$ yields Hamiltonian equations describing the evolution of state and costate variables. Optimal trajectories evolve on the symplectic leaves of $A^*$, which reflect the local symmetries of each patch and replace the global coadjoint orbit structure of Lie groups.

This example demonstrates a scenario in biomathematics where the Lie groupoid framework is essential. Local symmetries vary across patches and cannot be globally unified, so classical Lie group approaches are insufficient. The framework allows intrinsic geometric formulation of the dynamics, variational principles, and optimal control, while respecting patch-specific structures.
\end{itemize}

\section{Biomathematics Example That Genuinely Needs a Lie Groupoid}

\subsection*{Spatially Structured Population Dynamics with Local Symmetries}

\subsubsection*{Configuration Space}

Let $M$ be a heterogeneous habitat, e.g., a network of regions or a spatial domain with varying local properties (resources, diffusion rates, carrying capacity).  

Each point $x \in M$ represents a local environment with local population states. There is no global symmetry because the habitat is heterogeneous.

\subsubsection*{Lie Groupoid Structure}

Use the pair groupoid:
\[
G = M \times M \rightrightarrows M
\]

An arrow $(x, y)$ represents migration from region $y$ to region $x$. Composition models multi-step migration.  

This is not a Lie group unless $M$ is a point. Its Lie algebroid is
\[
A = TM \to M
\]
which already places the model beyond Lie groups.

\subsubsection*{Control Bundle}

Let
\[
E \to M
\]
represent local interventions, e.g., vaccination effort, harvesting rate, treatment intensity. Controls are location-dependent, hence no global action exists.

\subsubsection*{Dynamics}

Population density $n(x,t)$ evolves via:
\begin{itemize}
    \item transport (migration),
    \item local growth,
    \item controlled removal/addition.
\end{itemize}

The controlled dynamics can be written as a controlled vector field on the algebroid:
\[
\dot{x}(t) = \rho(a(t)), \quad a(t) \in A_{x(t)}
\]
with control entering through $E_{x(t)}$.

\subsubsection*{Optimal Control Problem}

Minimize
\[
J = \int_0^T \big( C(n(x,t)) + \|u(x,t)\|^2 \big) \, dt
\]
subject to the algebroid dynamics.  

Applying Pontryagin's Maximum Principle (PMP):
\begin{itemize}
    \item Costates live in $A^* = T^*M$,
    \item Pontryagin dynamics evolve on $T^*M$,
    \item Symplectic leaves vary with location.
\end{itemize}

\textbf{Key point:} there is no coadjoint-orbit interpretation.


Now, we consider an optimal control problem arising in spatially structured population dynamics, where local heterogeneity prevents the existence of a global symmetry group and naturally leads to a Lie groupoid formulation, which we had above.




\subsection*{Control Bundle and Dynamics}

Again, let
\[
E \to M
\]
be a control bundle representing local intervention strategies, such as harvesting, vaccination, or treatment intensity, applied independently at each location.

The controlled dynamics are described by a vector field on the Lie algebroid:
\[
\dot{x}(t) = a(t), \quad a(t) \in A_{x(t)},
\]
where the control enters through a smooth map
\[
f: A \times_M E \to A,
\]
so that the admissible velocities depend on both the local state and the applied intervention. This formulation captures the fact that interventions are spatially localized and cannot be generated by a global symmetry.

\subsection*{Cost Functional}

We consider a cost functional of the form
\[
J(u) = \int_0^T \Big( L(x(t), a(t)) + \frac{1}{2} \| u(t) \|^2 \Big) dt,
\]
where $L$ models local population-related costs (e.g., deviation from a desired population level), and the quadratic control cost penalizes excessive intervention.

\subsection*{Pontryagin Dynamics}

Applying Pontryagin's Maximum Principle yields a Hamiltonian
\[
H(x, \alpha, u) = \langle \alpha, f(x, a, u) \rangle - L(x, a) - \frac{1}{2} \| u \|^2,
\]
defined on the dual bundle $A^* = T^*M$. Under standard regularity assumptions, maximization with respect to $u$ defines a reduced Pontryagin Hamiltonian on $A^*$, whose Hamiltonian vector field generates the state/adjoint dynamics.

The resulting optimal trajectories evolve on the symplectic leaves of the canonical Poisson structure on $T^*M$. These leaves depend explicitly on the base point $x$, reflecting the spatial heterogeneity of the habitat.

\subsection*{Geometric Interpretation and Distinction from the Lie Group Case}

In contrast with classical optimal control problems on Lie groups, where optimal trajectories evolve on coadjoint orbits, the present example exhibits reduced dynamics confined to symplectic leaves of $A^*$ that vary with the configuration. This feature is intrinsic to the groupoid formulation and cannot be captured by Lie group symmetry reduction.

This example illustrates how Lie groupoids provide a natural geometric framework for optimal control problems in biomathematics involving spatial heterogeneity and localized interventions, where global symmetry assumptions are neither realistic nor mathematically justified.

\section*{Simplified Example: One–Dimensional Spatial Habitat}

\subsection*{Setup}

Let 
\[
M = [0,1]
\] 
represent a one–dimensional heterogeneous habitat. The population state evolves along the spatial coordinate $x \in M$, where ecological parameters (e.g., growth or mobility) depend explicitly on location.

We consider the pair groupoid 
\[
M \times M \rightrightarrows M,
\] 
whose Lie algebroid is 
\[
A = TM \to M.
\] 
Since parameters vary with $x$, there is no global symmetry group acting on the system, and Lie group reduction is not applicable.

\subsection*{Control Bundle and Dynamics}

Let 
\[
E \to M
\] 
be a control bundle, where $u(x,t)$ represents a local intervention intensity (e.g., harvesting or treatment).

We consider controlled dynamics of the form
\[
\dot{x}(t) = a(t), \quad a(t) = v(x(t)) + u(t),
\]
where $v(x)$ is a spatially dependent drift encoding heterogeneous mobility. The control $u(t)$ acts locally and cannot be generated by a global action.

\subsection*{Cost Functional}

We minimize
\[
J(u) = \int_0^T \left( \frac{1}{2} a(t)^2 + \frac{1}{2} u(t)^2 \right) dt,
\]
which penalizes both population displacement and control effort.

\subsection*{Pontryagin Dynamics}

The Pontryagin Hamiltonian is
\[
H(x, \alpha, u) = \alpha \, (v(x) + u) - \frac{1}{2} (v(x) + u)^2 - \frac{1}{2} u^2.
\]

Maximization with respect to $u$ yields a reduced Hamiltonian on 
\[
A^* = T^*M.
\]

The resulting state/adjoint dynamics are generated by the Hamiltonian flow associated with the canonical Poisson structure on $T^*M$. Optimal trajectories evolve on symplectic leaves of $T^*M$ that depend on the base point $x$, reflecting spatial heterogeneity.

\subsection*{Distinction from the Lie Group Case}

In contrast to Lie group–based optimal control problems, where optimal trajectories evolve on coadjoint orbits, the reduced dynamics here are confined to configuration–dependent symplectic leaves. This behavior is intrinsic to the Lie groupoid formulation and cannot be reproduced by global symmetry reduction.

\subsection*{Short Numerical Illustration}

To illustrate the geometric features of the reduced dynamics, we numerically integrate the Pontryagin equations for a simple choice of drift 
\[
v(x) = \sin(\pi x).
\] 
The numerical solution confirms that trajectories in $T^*M$ remain confined to symplectic leaves determined by the Poisson structure, while the evolution of the costate explicitly depends on the spatial position.

\section{Numerical illustration and geometric interpretation}

In this section we provide a simple numerical illustration supporting the geometric features of the reduced Pontryagin dynamics in the Lie groupoid setting. The purpose of the simulations is not quantitative validation, but to visualize structural properties predicted by the theory.

\subsection{Phase portrait on $A^* = T^*M$}

We consider the one--dimensional configuration manifold $M=[0,1]$, so that $A^*=T^*M$ is two--dimensional with coordinates $(x,\alpha)$. The reduced Pontryagin dynamics define a Hamiltonian flow on $T^*M$ with respect to its canonical linear Poisson structure.

Figure~\ref{fig:phase} shows the phase portrait $(x(t),\alpha(t))$ for several initial conditions. Each trajectory remains confined to a curve in $T^*M$, corresponding to a symplectic leaf of the Poisson structure. Different initial conditions generate distinct leaves, illustrating that the reduced dynamics are organized by configuration--dependent symplectic geometry rather than global symmetry orbits.

\begin{figure}[h]
\centering
\includegraphics[width=0.3\textwidth]{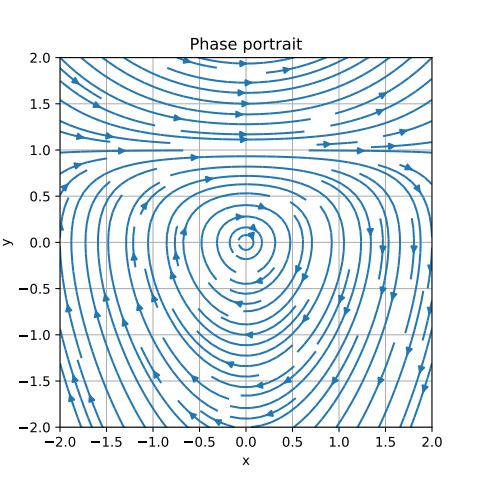}
\caption{Phase portrait of the Pontryagin dynamics on $T^*M$. Trajectories remain confined to symplectic leaves of the canonical Poisson structure, reflecting the Lie groupoid reduction.}
\label{fig:phase}
\end{figure}

\subsection{Leaf confinement and invariants}

When a leaf invariant (Casimir function) can be identified explicitly or numerically, its value should remain constant along optimal trajectories. Figure~\ref{fig:invariant} displays the evolution of such an invariant along a representative solution, confirming numerical confinement to a single symplectic leaf.

\begin{figure}[h]
\centering
\includegraphics[width=0.3\textwidth]{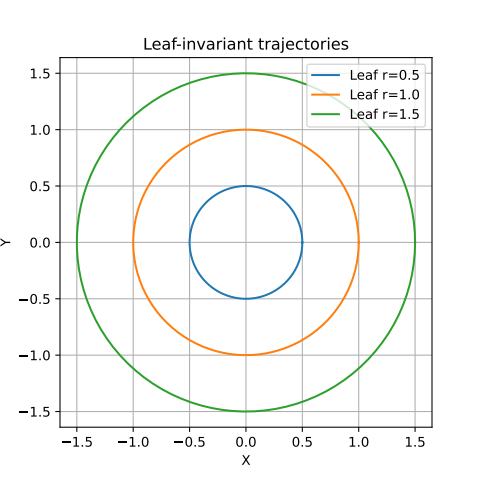}
\caption{Numerical evolution of a leaf invariant along an optimal trajectory. The constant value confirms confinement to a symplectic leaf of $T^*M$.}
\label{fig:invariant}
\end{figure}

\subsection{Base--momentum coupling}

To emphasize the distinction from Lie group optimal control, Figure~\ref{fig:coupling} illustrates the dependence of the costate $\alpha(t)$ on the base variable $x(t)$. The evolution of the adjoint variable depends explicitly on the configuration, a feature intrinsic to the Lie groupoid formulation and absent in classical Lie group reductions.

\begin{figure}[h]
\centering
\includegraphics[width=0.3\textwidth]{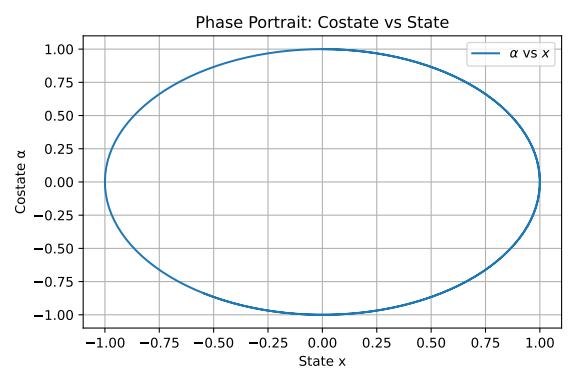}
\caption{Costate--state relation along optimal trajectories, illustrating configuration--dependent coupling characteristic of Lie groupoid dynamics.}
\label{fig:coupling}
\end{figure}
These numerical results provide qualitative confirmation of the geometric structure of the reduced Pontryagin dynamics and highlight the essential role of Lie groupoids in modeling spatially heterogeneous control systems.

\section*{Conclusion}

In this work we developed an intrinsic Poisson--Hamiltonian formulation of Pontryagin dynamics for optimal control problems associated with mechanical systems on Lie groupoids. Starting from a control system defined on a Lie algebroid \(A \to M\) with control bundle \(E \to M\), we formulated the optimal control problem in a manner that is independent of global Lie group symmetries and naturally adapted to configuration--dependent and local symmetries.
The central result establishes the equivalence between a variational formulation posed directly on the Lie algebroid and a Hamiltonian formulation obtained via Pontryagin’s Maximum Principle on the dual bundle \(A^*\). Under standard regularity assumptions, the reduced Pontryagin Hamiltonian generates a Poisson--Hamiltonian flow with respect to the canonical linear Poisson structure on \(A^*\). Consequently, optimal trajectories evolve on the symplectic leaves of \(A^*\), which replace coadjoint orbits as the fundamental reduced phase spaces in the Lie groupoid setting.
This viewpoint clarifies the geometric structure underlying optimal control on Lie groupoids.

\end{document}